\documentclass[3p]{elsarticle}

\usepackage{geometry}

\usepackage{amsmath,amsthm,amssymb,mathrsfs}
\usepackage{enumerate}
\usepackage{slashed}
\usepackage{txfonts,dsfont}
\usepackage{aliascnt}
\usepackage[breaklinks,colorlinks]{hyperref}
\usepackage{comment}

\makeatletter
\newcommand{\autorefcheckize}[1]{%
  \expandafter\let\csname @@\string#1\endcsname#1%
  \expandafter\DeclareRobustCommand\csname relax\string#1\endcsname[1]{%
    \csname @@\string#1\endcsname{##1}\wrtusdrf{##1}}%
  \expandafter\let\expandafter#1\csname relax\string#1\endcsname
}
\makeatother

\theoremstyle{plain}

\newtheorem{theorem}{Theorem}[section]

\newaliascnt{lem}{theorem}
\newtheorem{lem}[lem]{Lemma}
 \aliascntresetthe{lem}

\newaliascnt{cor}{theorem}
\newtheorem{cor}[cor]{Corollary}
 \aliascntresetthe{cor}

\newaliascnt{prop}{theorem}

 \aliascntresetthe{prop}

\theoremstyle{remark}

\newtheorem{rem}{Remark}[section]

\theoremstyle{definition}

\numberwithin{equation}{section}

\newcommand{\norm}[1]{\left\lVert#1\right\rVert}
\newcommand{\abs}[1]{\left\lvert#1\right\rvert}
\newcommand{\set}[1]{\left\{#1\right\}}

\newcommand*{\Rmn}[1]{\uppercase\expandafter{\romannumeral#1}}
\newcommand*{\dif}{\mathop{}\!\mathrm{d}}

\journal{XXX}

\allowdisplaybreaks

\begin{document}

\begin{frontmatter}

\title{Existence of Kazdan-Warner equation with sign-changing prescribed function\tnoteref{SZ}}

\author[whu1,whu2]{Linlin Sun}
\address[whu1]{School of Mathematics and Statistics, Wuhan University, Wuhan 430072, China}
\address[whu2]{Hubei Key Laboratory of Computational Science, Wuhan University, Wuhan, 430072, China}
\ead{sunll@whu.edu.cn}

\author[mpi]{Jingyong Zhu\texorpdfstring{\corref{zjy}}{}}
\address[mpi]{Max Planck Institute for Mathematics in the Sciences, Inselstrasse 22, 04103 Leipzig, Germany}
\ead{jizhu@mis.mpg.de}


\cortext[zjy]{Corresponding author.}

\tnotetext[SZ]{This research is partially supported by the National Natural Science Foundation of China (Grant No. 11971358, 11801420). The first author would like to thank Prof. Chen Xuezhang for his useful suggestions and support.}

\begin{abstract}

In this paper, we study the following Kazdan-Warner equation with sign-changing prescribed function $h$
\begin{align*}
    -\Delta u=8\pi\left(\dfrac{he^{u}}{\int_{\Sigma}he^{u}}-1\right)
\end{align*}
on a closed Riemann surface whose area equals one. The solutions are the critical points of the functional $J_{8\pi}$ which is defined by
\begin{align*}
    J_{8\pi}(u)=\dfrac{1}{16\pi}\int_{\Sigma}\abs{\nabla u}^2+\int_{\Sigma}u-\ln\abs{\int_{\Sigma}he^{u}},\quad u\in H^1\left(\Sigma\right).
\end{align*}
We prove the existence of minimizer of $J_{8\pi}$ by assuming 
\begin{equation*}
    \Delta \ln h^++8\pi-2K>0
    \end{equation*}
    at each maximum point of $2\ln h^++A$, where $K$ is the Gaussian curvature, $h^+$ is the positive part of $h$ and $A$ is the regular part of the Green function. This generalizes the existence result of Ding, Jost, Li and Wang [Asian J. Math. 1(1997), 230-248] to the sign-changing prescribed function case. We are also interested in the blow-up behavior of a sequence $u_{\varepsilon}$ of critical points of $J_{8\pi-\varepsilon}$ with $\int_{\Sigma}he^{u_{\varepsilon}}=1, \lim\limits_{\varepsilon\searrow   0}J_{8\pi-\varepsilon}\left(u_{\varepsilon}\right)<\infty$  and obtain the following identity during the blow-up process
\begin{equation*}
    -\varepsilon=\frac{16\pi}{(8\pi-\varepsilon)h(p_\varepsilon)}\left[\Delta \ln h(p_\varepsilon)+8\pi-2K(p_\varepsilon)\right]\lambda_{\varepsilon}e^{-\lambda_{\varepsilon}}+O\left(e^{-\lambda_{\varepsilon}}\right),
\end{equation*}
where $p_\varepsilon$ and $\lambda_\varepsilon$ are the maximum point and maximum value of $u_\varepsilon$, respectively. Moreover, $p_{\varepsilon}$ converges to the   blow-up point which is a critical point of the function $2\ln h^{+}+A$. 
\end{abstract}

\begin{keyword}
Kazdan-Warner equation \sep sign-changing prescribed function \sep existence.

 \MSC[2020] 35B33 \sep 	58J05

\end{keyword}

\end{frontmatter}


\section{Introduction}
Let $\Sigma$ be a closed Riemann surface  whose area equals one. Let $h$ be a nonzero smooth function on $\Sigma$ such that $\max\limits_{\Sigma}h>0$. For each positive number $\rho$, we consider the following functional
\begin{align*}
    J_{\rho}(u)=\dfrac{1}{2\rho}\int_{\Sigma}\abs{\nabla u}^2+\int_{\Sigma}u-\ln\abs{\int_{\Sigma}he^{u}},\quad u\in H^1\left(\Sigma\right).
\end{align*}
 The critical points of $J_{\rho}$ are solutions to  the following mean field equation
\begin{align}\label{eq:KW}
    -\Delta u=\rho\left(\dfrac{he^{u}}{\int_{\Sigma}he^{u}}-1\right)
\end{align}
where $\Delta$ is the Laplace operator on $\Sigma$.

Mean field equation has a strong relationship with Kazdan-Warner equation. Forty years ago, Kazdan and Warner \cite{kazdan1974curvature} considered the solvability of the equation
\begin{align*}
    -\Delta u=he^u-\rho,
\end{align*}
where $\rho$ is a constant and $h$ is some smooth prescribed function. When $\rho>0$, the equation above is equivalent to the mean field equation \eqref{eq:KW}. 
The special case $\rho=8\pi$ is sometimes called the Kazdan-Warner equation. In particular, when $\Sigma$ is the standard sphere $\mathbb{S}^2$, it is called the Nirenberg problem, which comes from the conformal geometry. It has been studied by Moser \cite{moser1971sharp}, Kazdan and Warner \cite{kazdan1974curvature}, Chen
and Ding \cite{chen1987scalar}, Chang and Yang \cite{chang1987prescribing} and others. The mean field equation \eqref{eq:KW} appears in various context such as the abelian Chern-Simons-Higgs models.  The existence of solutions
of \eqref{eq:KW} and its evolution problem has been widely studied in recent decades (see for example \cite{BreMer91uniform,casteras2015mean,CheLin02sharp,chen2003topological,DinJosLiWan97differential,ding1999existence,djadli2008existence,DM2008existence,LiZhu19convergence,Lin00topological,Mal08morse,struwe2002curvature,SunZhu20global} and the references therein).

In this paper, we consider the existence theory of Kazdan-Warner equation ($\rho=8\pi$) with sign-changing prescribed function. The key is to analyze the asymptotic behavior of the blow-up solutions $u_\varepsilon$ (see \eqref{blow-up sequence}) and the functional $J_{8\pi}$. We prove the following identity near the blow-up point, whose analogue was proved by Chen and Lin in \cite{CheLin02sharp} when the prescribed function $h$ is positive. 

\begin{theorem}\label{blow-up id}
Let $h$ be positive somewhere on $\Sigma$ and $u_{\varepsilon}$ a blow-up sequence satisfying
\begin{align}\label{blow-up sequence}
    -\Delta u_{\varepsilon}=\left(8\pi-\varepsilon\right)\left(he^{\varepsilon}-1\right),\quad\text{in}\ \Sigma
\end{align}
and
\begin{align}\label{eq:upper-J}
    \lim_{\varepsilon\searrow 0}J_{8\pi-\varepsilon}\left(u_{\varepsilon}\right)<\infty.
\end{align}
Then up to a subsequence, for $p_{\varepsilon}\in\Sigma$ with
\begin{align*}
    \lambda_{\varepsilon}=\max_{\Sigma}u_{\varepsilon}=u_{\varepsilon}\left(p_{\varepsilon}\right),
\end{align*}
we have
\begin{equation*}
    -\varepsilon=\frac{16\pi}{(8\pi-\varepsilon)h(p_\varepsilon)}[\Delta \ln h^+(p_\varepsilon)+8\pi-2K(p_\varepsilon)]\lambda_{\varepsilon}e^{-\lambda_{\varepsilon}}+O\left(e^{-\lambda_{\varepsilon}}\right),
\end{equation*}
where $K$ denotes the Gaussian curvature of $\Sigma$.
\end{theorem}

This yields a uniform bound of minimizers as $\varepsilon\searrow 0$  provided that $\Delta \ln h^++8\pi-2K>0$ at all blow-up points. Let $G(q,p)$ be the Green function on $\Sigma$ with singularity at $p$, i.e.,
\begin{align*}
    \Delta G(\cdot,p)=1-\delta_{p},\quad\int_{\Sigma}G(\cdot,p)=0.
\end{align*}
Under a local normal coordinate $x$ centering at $p$, we have
\begin{equation}\label{expansion G}
8\pi G(x,p)=-4\ln\abs{x}+A(p)+b_1x_1+b_2x_2+c_1x_1^2+2c_2x_1x_2+c_3x_2^2+O\left(\abs{x}^3\right).
\end{equation}
By \autoref{cp position}, we know the blow-up point has to be a critical point of $2\ln h^+(p)+A(p)$. Thus, we get an  existence result.
\begin{cor}
Let $\Sigma$ be a compact Riemann surface and $K(p)$ be its Gaussian curvature. Suppose $h(p)$ is a smooth function which is positive somewhere on $\Sigma$. If we have the following for all critical points of $2\ln h^++A$ 
\begin{equation*}
    \Delta \ln h^++8\pi-2K>0,
\end{equation*}
then equation \eqref{eq:KW} has a solution for $\rho=8\pi$.
\end{cor}

Furthermore, if $u_{\varepsilon}$ is a minimizer of $J_{8\pi-\varepsilon}$, we can show the blow-up point is actually the maximum point of $2\ln h^++A$.

\begin{theorem}\label{functional value}
 If $u_{\varepsilon}$ is a minimizer of $J_{8\pi-\varepsilon}$ and blows up as $\varepsilon\searrow 0$, then the blow-up point $p_0$ is a maximum point of the function $2\ln h^++A$. Moreover,
\begin{align*}
    \inf_{u\in H^1\left(\Sigma\right)}J_{8\pi}=-1-\ln\pi-\left(\ln h(p_0)+\frac12 A(p_0)\right),
\end{align*}
and there is a sequence $\phi_{\varepsilon}\in H^1\left(\Sigma\right)$ such that
\begin{align*}
\begin{split}
    J_{8\pi}\left(\phi_{\varepsilon}\right)&=-1-\ln\pi-\left(\ln h(p_0)+\frac12 A(p_0)\right)\\
    &\quad-\dfrac{1}{4}\left(\Delta\ln h(p_0)+8\pi-2K(p_0)\right)\varepsilon\ln\varepsilon^{-1}
    +o\left(\varepsilon\ln\varepsilon^{-1}\right).
    \end{split}
\end{align*}
\end{theorem}

Hence, we obtain a minimizing solution of the functional $J_{8\pi}$. 
 
\begin{theorem}\label{thm:DJLW}
Let $\Sigma$ be a compact Riemann surface and $K$ be its Gaussian curvature. Suppose $h$ is a smooth function which is positive somewhere on $\Sigma$. If the following holds at the maximum points of $2\ln h^++A$ 
\begin{equation*}
    \Delta \ln h+8\pi-2K>0,
\end{equation*}
then equation \eqref{eq:KW} has a minimizing solution for $\rho=8\pi$.
\end{theorem}

\begin{rem}The condition mentioned in \autoref{thm:DJLW} can not hold on $2$-sphere with arbitrary metric. 
Assume $g=e^{2\phi}g_0$ and solve
\begin{align*}
    -\Delta_{g_0}\psi=\dfrac{1}{\abs{\Sigma}_{g_0}}-\dfrac{e^{2\phi}}{\abs{\Sigma}_{g}},\quad \int_{\Sigma}\psi\dif\mu_{g_0}=0,
\end{align*}
where $\abs{\Sigma}_g$ stands for the area of $\Sigma$ with respect to the metric $g$.
Set $h_0=he^{2\phi+\rho \psi}$. Then
\begin{align*}
    J_{\rho,h,g}(u)=J_{\rho,h_0,g_0}\left(u-\rho \psi\right)-\dfrac{\rho}{2}\int_{\Sigma}\abs{\dif \psi}_{g_0}^2\dif\mu_{g_0}.
\end{align*}
If the condition mentioned in \autoref{thm:DJLW} holds, then there is a minimizer of $J_{8\pi,h,g}$. Hence, there is also a minimizer of $J_{8\pi, h_0,g_0}$. If $\Sigma$ is a $2$-sphere, we choose $g_0$ such that the Gaussian curvature is constant, then $h_0$ must be a constant (see \cite{Han90prescribing}). Thus $h$ is a positive function and
\begin{align*}
    \Delta_{g}\ln h+\dfrac{8\pi}{\abs{\Sigma}_{g}}-2K_{g}=e^{-2\phi}\left(\Delta_{g_0}\ln h_0+\dfrac{8\pi}{\abs{\Sigma}_{g_0}}-2K_{g_0}\right)=0
\end{align*}
which is a contradiction. 
\end{rem}

\begin{rem}Zhu \cite{Zhu18generalized} also obtained the infimum of the functional $J_{8\pi}$ if there is no minimzer (when $h$ is non-negative). He pointed out the blow-up point must be the positive point of $h$ and used the maximum principle to estimate the lower bound of the functional $J_{8\pi}$ when $h$ is non-negative. In our case, the maximum principle does not work since $h$ is sign-changed. We will use the method of energy estimate to give the lower bound of the functional $J_{8\pi}$. Such a method also can be used to consider the flow case (cf. \cite{SunZhu20global,LiZhu19convergence}) and the Palais-Smale sequence.
\end{rem}

\begin{rem}
The method in the proof of \autoref{thm:DJLW} can be used to prove the convergence of the Kazdan-Warner flow. In other words, under the same condition mentioned in \autoref{thm:DJLW}, there exists an initial date $u_0$ such that the following flow
\begin{align*}
    \dfrac{\partial u}{\partial t}=\Delta u+8\pi\left(\dfrac{he^{u}}{\int_{\Sigma}he^{u}}-1\right),\quad u(0)=u_0
\end{align*}
converges to a minimizer of $J_{8\pi}$. This gives a generalization of the  previous results \cite{LiZhu19convergence} (positive prescribed function case) and \cite{SunZhu20global} (non-negative prescribed function case). Rencently, Chen, Li, Li and Xu \cite{CheLiLiXu20gaussian} consider another flow approach to the Gaussian curvature flow on sphere and reproved the existence result for sign-changing prescribed function which was obtained by Han \cite{Han90prescribing}.
\end{rem}

\section{Preliminary}

Recall the strong Trudinger-Moser inequality (cf. \cite[Theorem 1.7]{Fon93sharp})
\begin{align*}
    \sup_{u\in H^1\left(\Sigma\right), \int_{\Sigma}\abs{\nabla u}^2\leq 1, \int_{\Sigma}u=0}\int_{\Sigma}\exp\left(4\pi u^2\right)<\infty.
\end{align*}
which implies the Trudinger-Moser inequality
\begin{align}\label{eq:TM}
    \ln\int_{\Sigma}e^{u}\leq\dfrac{1}{16\pi}\int_{\Sigma}\abs{\nabla u}^2+\int_{\Sigma}u+c
\end{align}
where $c$ is a uniform constant depends only the geometry of $\Sigma$. 

We may assume $h$ is positive somewhere. If $0<\rho<8\pi$, then applying the Trudinger-Moser inequality \eqref{eq:TM} Kazdan and Warner (\cite[Theorem 7.2]{kazdan1974curvature}) proved that the Kazdan-Warner equation \eqref{eq:KW} admits a solution $u$ which minimizes the functional $J_{\rho}$ and satisfies
\begin{align*}
    \int_{\Sigma}he^{u}=1.
\end{align*}

We consider the critical case $\rho=8\pi$. For every $\varepsilon\in(0,8\pi)$, let $u_{\varepsilon}$ be a minimizer of $J_{8\pi-\varepsilon}$ which satisfies
\begin{align*}
    \int_{\Sigma}he^{u_{\varepsilon}}=1.
\end{align*}
Thus $u_{\varepsilon}$ satisfies \eqref{blow-up sequence}.
It is clear that the function
\begin{align*}
    \rho\mapsto\inf_{u\in H^1\left(\Sigma\right)}J_{\rho}(u)
\end{align*}
is a decreasing function on $(0,+\infty)$. In particular, $u_{\varepsilon}$ satisfies \eqref{eq:upper-J}. By the Trudinger-Moser inequality \eqref{eq:TM}, we have
\begin{align}\label{eq:lower-J}
    J_{8\pi-\varepsilon}\left(u_{\varepsilon}\right)\geq&\ln\int_{\Sigma}e^{u_{\varepsilon}}-c.
\end{align}
Thus \eqref{eq:upper-J} and \eqref{eq:lower-J} gives
\begin{align}\label{eq:upper-energy}
    \int_{\Sigma}e^{u_{\varepsilon}}\leq C,\quad\forall \varepsilon\in(0,4\pi).
\end{align}
One can check that
\begin{align*}
    \lim_{\varepsilon\to0}J_{8\pi}\left(u_{\varepsilon}\right)=\inf_{u\in H^1\left(\Sigma\right)}J_{8\pi}(u).
\end{align*}
If 
\begin{align*}
    \limsup_{\varepsilon\to0}\max_{\Sigma}u_{\varepsilon}<+\infty,
\end{align*}
then up to a subsequence $u_{\varepsilon}$ converges smoothly to a minimizer of $J_{8\pi}$.

In the rest of this section, we only assume $u_{\varepsilon}$ is a solution to \eqref{blow-up sequence} and satisfies the condition \eqref{eq:upper-energy}. 

Assume now $\set{u_{\varepsilon}}$ is a blow-up sequence, i.e.,
\begin{align*}
    \limsup_{\varepsilon\to0}\max_{\Sigma}u_{\varepsilon}=+\infty.
\end{align*}
Without loss of generality, we may assume $h^{\pm}e^{u_{\varepsilon}}\dif\mu_{\Sigma}$ converges to a nonzero Radon measure $\mu^{\pm}$ as $\varepsilon\to0$.
Define the singular set $S$ of the sequence $\set{u_{\varepsilon}}$ by
\begin{align*}
    S=\set{x\in\Sigma: \abs{\mu}\left(\set{x}\right)\geq \dfrac12},
\end{align*}
where $\abs{\mu}=\mu^++\mu^-$.
It is clear that $S$ is a finite set. Applying Brezis-Merle’s estimate \cite[Theorem 1]{BreMer91uniform}, one can obtain that for each compact subset $K\subset \Sigma\setminus S$ (cf. \cite[Lemma 2.8]{DinJosLiWan97differential})
\begin{align}\label{eq:out}
    \norm{u_{\varepsilon}-\int_{\Sigma} u_{\varepsilon}}_{L^{\infty}\left(K\right)}\leq C_{K}.
\end{align}
Then one obtain a characterization of $S$ by the blow-up sets of $\set{u_{\varepsilon}}$ (cf. \cite[Page 1240]{BreMer91uniform})
\begin{align*}
    S=\set{p\in\Sigma: \exists \ p_\varepsilon\in\Sigma,\ s.t.\ \lim_{\varepsilon\to0}p_\varepsilon=p,\   \lim_{\varepsilon\to0}u_{\varepsilon}\left(p_\varepsilon\right)=\infty.}
\end{align*}
Moreover, $S$ is nonempty and 
\begin{align*}
    \lim_{\varepsilon\to0}\int_{\Sigma} u_{\varepsilon}=-\infty,
\end{align*}
which implies that $u_{\varepsilon}$ goes to $-\infty$ uniformly on each compact subsets $K\subset\Sigma\setminus S$. Thus, $\abs{\mu}$ is a Dirac measure. 
By using blow-up analysis (cf. \cite[Lemma 1]{LiSha94blowup}) together with the classification result of Chen-Li \cite[Theorem 1]{CheLi91classification}, one can show that $\mu^{-}=0$ and
\begin{align*}
    S=\set{p\in\Sigma : \mu^+\left(\set{p}\right)\geq 1, h(p)>0.}
\end{align*}
Notice that $he^{u_{\varepsilon}}\dif\mu_{\Sigma}$ converges to the nonzero Radon measure $\mu^+$ as $\varepsilon\to0$. We conclude that $S=\set{p_0}$ is a single point set and $\abs{\mu}=\mu^+=\delta_{x_0}$. Thus
\begin{lem}[cf. Lemma 2.6 in \cite{DinJosLiWan97differential}]\label{convergence}
$u_\varepsilon-\int_{\Sigma}u_\varepsilon$ converges to $8\pi G(\cdot,p_0)$ weakly in $W^{1,q}\left(\Sigma\right)$ and strongly in $L^q\left(\Sigma\right)$ for every $q\in(1,2)$, and converges in $C^2_{loc}\left(\Sigma\setminus\set{p_0}\right)$.
\end{lem}

For a fixed small $\delta_0>0$ and $u_{\varepsilon}$ of $J_{8\pi}$, we define $\rho_\varepsilon$ to be 
\begin{equation*}
    \rho_\varepsilon=(8\pi-\varepsilon)\int_{B_{\delta_0}(p_0)}he^{u_\varepsilon}
\end{equation*}
and 
\begin{align*}
    \lambda_{\varepsilon}=u_{\varepsilon}(p_\varepsilon)=\max_{\overline{B_{\delta}(p_0)}}u_{\varepsilon}\to +\infty.
\end{align*}
We may assume
\begin{align*}
    h\vert_{B_{\delta_0}(p_0)}\geq\dfrac12h(p_0)>0,\quad \max_{\partial B_{\delta_0}(p_0)}u_{\varepsilon}-\min_{\partial B_{\delta_0}(p_0)}u_{\varepsilon}\leq C,\quad\int_{B_{\delta_0}(p_0)}e^{u_{\varepsilon}}\leq C.
\end{align*}
Li \cite[Theorem 0.3]{Li99harnack} obtained the following local estimate
\begin{equation}\label{eq:local}
\abs{u_\varepsilon(p)-\ln{\frac{e^{\lambda_\varepsilon}}{1+\frac{(8\pi-\varepsilon)h_{p_\varepsilon}}{8}e^{\lambda_\varepsilon}|p-p_\varepsilon|^2}}}\leq C
\end{equation}
for $p\in B_{\delta_0}(p_0)$, where $|p-p_\varepsilon
|$ stands for the distance between $p$ and $p_\varepsilon$. Together with \autoref{convergence}, the above local estimate \eqref{eq:local} gives the following
\begin{lem}[cf. Corollary 2.4 in \cite{CheLin02sharp}]\label{outside}
There exists a constant $C>0$ such that
\begin{equation*}
    \abs{u_\varepsilon+\lambda_\varepsilon}\leq C \ \text{in} \ \Sigma\setminus B_{\delta_0}(p_0).
\end{equation*}
\end{lem}

\begin{lem}[cf. Estimate A in \cite{CheLin02sharp}]\label{estimate B}
Set $w_\varepsilon$ to be the error term defined by
\begin{equation*}
    \omega_\varepsilon(q)=u_\varepsilon(q)-\rho_\varepsilon G(q,p_\varepsilon)-\bar{u}_\varepsilon
\end{equation*}
on $\Sigma\setminus B_{\delta_0/2}(p_0)$. Then we have
\begin{equation*}
    \norm{\omega_\varepsilon}_{C^1\left(\Sigma\setminus B_{\delta_0}(p_0)\right)}=O\left(e^{-\lambda_\varepsilon/2}\right).
\end{equation*}
\end{lem}
\begin{proof}
Notice that $h$ maybe non-positive outside of $B_{\delta_0/2}(p_0)$ and in this case we also have the above estimate. We list a proof here. By Green representation formula, for every $q\in\Sigma\setminus B_{\delta_0}(p_0)$
\begin{align*}
    u_{\varepsilon}(q)-\bar u_{\varepsilon}=&\left(8\pi-\varepsilon\right)\int_{\Sigma}G(q,p)\left[h(p)e^{u_{\varepsilon}(p)}-1\right]\dif\mu_{\Sigma}(p)\\
    =&\left(8\pi-\varepsilon\right)\int_{\Sigma}\left(G(q,p)-G\left(q,p_{\varepsilon}\right)\right)\left[h(p)e^{u_{\varepsilon}(p)}-1\right]\dif\mu_{\Sigma}(p)\\
    =&\left(8\pi-\varepsilon\right)\int_{\Sigma\setminus B_{\delta_0/2}(p_0)}\left(G(q,p)-G\left(q,p_{\varepsilon}\right)\right)h(p)e^{u_{\varepsilon}(p)}\dif\mu_{\Sigma}(p)\\
    &+\left(8\pi-\varepsilon\right)\int_{ B_{\delta_0/2}(p_0)}\left(G(q,p)-G\left(q,p_{\varepsilon}\right)\right)h(p)e^{u_{\varepsilon}(p)}\dif\mu_{\Sigma}(p)+(8\pi-\varepsilon)G(q,p_{\varepsilon})\\
    =&(8\pi-\varepsilon)G(q,p_{\varepsilon})+O\left(e^{-\lambda_{\varepsilon}/2}\right).
\end{align*}
Here we used estimate \eqref{eq:out} and Li's local estimate \eqref{eq:local}.
By definition,
\begin{align*}
    \rho_{\varepsilon}=(8\pi-\varepsilon)-(8\pi-\varepsilon)\int_{\Sigma\setminus B_{\delta_0}(p_0)}he^{u_{\varepsilon}}=(8\pi-\varepsilon)+O\left(e^{-\lambda_{\varepsilon}}\right).
\end{align*}
Thus
\begin{align*}
    u_{\varepsilon}(q)-\bar u_{\varepsilon}-\rho_{\varepsilon}G(q,p_{\varepsilon})=O\left(e^{-\lambda_{\varepsilon}/2}\right),\quad\forall q\in \Sigma\setminus B_{\delta_0}(p_0).
\end{align*}
Notice that
\begin{align*}
    -\Delta\left(u_{\varepsilon}-\bar u_{\varepsilon}-\rho_{\varepsilon}G(\cdot,p_{\varepsilon})\right)=(8\pi-\varepsilon)he^{u_{\varepsilon}}+\rho_{\varepsilon}-(8\pi-\varepsilon)=O\left(e^{-\lambda_{\varepsilon}}\right),\quad\text{in}\ \Sigma\setminus B_{\delta_0}(p_0)
\end{align*}
and
\begin{align*}
    u_{\varepsilon}-\bar u_{\varepsilon}-\rho_{\varepsilon}G(\cdot,p_{\varepsilon})=O\left(e^{-\lambda_{\varepsilon}/2}\right),\quad\text{on}\ \partial B_{\delta_0}(p_0).
\end{align*}
The standard elliptic estimate gives
\begin{align*} 
    \norm{u_{\varepsilon}-\bar u_{\varepsilon}-\rho_{\varepsilon}G(\cdot,p_{\varepsilon})}_{C^1\left(\Sigma\setminus B_{\delta_0}(p_0)\right)}=O\left(e^{-\lambda_{\varepsilon}/2}\right).
\end{align*}
\end{proof}

Based on these facts, we then have the following local estimates. The proofs are same as those in \cite{CheLin02sharp}, so we omit them here.

\begin{lem}[cf. Estimate B in \cite{CheLin02sharp}]\label{cp position}
By using the local normal coordinate $x$ centering at $p_{\varepsilon}$, we set the regular part of Green function $G(x,p_\varepsilon)$ to be
\begin{equation*}
    \tilde{G}_{\varepsilon}(x)=G(x,p_\varepsilon)+\dfrac{1}{2\pi}\ln\abs{x},
\end{equation*}
and set
\begin{equation*}
    G_{\varepsilon}^*(x)=\rho_\varepsilon\tilde{G}_\varepsilon(x).
\end{equation*}
Then we get
\begin{equation*}
    \abs{\nabla\left(\ln h^++G_{\varepsilon}^*\right)(p_{\varepsilon})}=O\left(e^{-\lambda_\varepsilon/2}\right).
\end{equation*}
\end{lem}
 Notice that the Green function is symmetric and we conclude that
\begin{align*}
    \abs{\nabla\left(2\ln h^++\dfrac{8\pi-\varepsilon}{8\pi}A\right)\left(p_{\varepsilon}\right)}=O\left(e^{-\lambda_{\varepsilon}/2}\right).
\end{align*}

In $B_{\delta_0}(p_\varepsilon)$, we define the following function as in \cite{CheLin02sharp}
\begin{equation*}
    v_\varepsilon(p)=\ln\dfrac{e^{\lambda_{\varepsilon}}}{\left(1+\frac{(8\pi-\varepsilon)h(p_\varepsilon)}{8}e^{\lambda_{\varepsilon}}|p-q_{\varepsilon}|^2\right)^2},
\end{equation*}
where $q_{\varepsilon}$ is chosen to satisfy
\begin{equation*}
    \nabla v_{\varepsilon}(p_{\varepsilon})=\nabla\ln h(p_{\varepsilon}),
\end{equation*}
which implies $\abs{p_\varepsilon-q_\varepsilon}=O\left(e^{-\lambda_\varepsilon}\right)$.
We also set the error term as
\begin{equation*}
    \eta_\varepsilon(p)=u_\varepsilon(p)-v_\varepsilon(p)-(G_{\varepsilon}^*(p)-G_{\varepsilon}^*(p_\varepsilon))
\end{equation*}
and 
\begin{equation*}
    R_{\varepsilon}=\left(\frac{(8\pi-\varepsilon)h(p_\varepsilon)}{8}e^{\lambda_{\varepsilon}}\right)^{\frac12}\delta_0.
\end{equation*}
Then we have the following estimate for the scaled function $\tilde{\eta}_\varepsilon(z)=\eta_\varepsilon\left(\delta_0R_{\varepsilon}^{-1}z\right)$ for $|z|\leq R_{\varepsilon}$.

\begin{lem}[cf. Estimates C, D and E in \cite{CheLin02sharp}]\label{remainder}
For any $\tau\in(0,1)$, there exists a constant $C=C_\tau$ such that
\begin{equation*}
    \eta_\varepsilon(p)=\left(4-\dfrac{\rho_\varepsilon}{2\pi}\right)\ln{|p-p_\varepsilon|}+O\left(\lambda_\varepsilon e^{-\frac{\tau\lambda_\varepsilon}{2}}\sup_{\frac{\delta_0}{2}\leq|p-p_\varepsilon|\leq\delta_0}|\eta_\varepsilon|+e^{-\frac{\lambda_\varepsilon}{2}}\right)
\end{equation*}
 and
\begin{equation*}
    \abs{\tilde{\eta}_\varepsilon(z)}\leq C\left(1+|z|\right)^\tau\left(e^{-\tau\lambda_\varepsilon}+e^{-\frac{\tau}{2}\lambda_\varepsilon}|8\pi-\rho_\varepsilon|\right)
\end{equation*}
hold for $p\in \bar{B}_{\delta_0}(p_\varepsilon)\setminus{B_{\delta_0/2}(p_\varepsilon)}$ and $|z|\leq R_{\varepsilon}$. 
\end{lem}
The following lemma shows the relationship between $\rho_\varepsilon-8\pi$ and $\eta_\varepsilon$.
\begin{lem}[cf. Estimate F in \cite{CheLin02sharp}]\label{boundary term}
\begin{equation*}
    \rho_\varepsilon-8\pi=-\int_{\partial B_{\delta_0}(p_\varepsilon)}\frac{\partial \eta_\varepsilon}{\partial\nu}d\sigma+O\left(e^{-\lambda_\varepsilon}\right),
\end{equation*}
where $\nu$ denotes the unit outer normal of $\partial B_{\delta_0}(p_\varepsilon)$.
\end{lem}

\section{Proof of \autoref{blow-up id}}
In this section, we prove \autoref{blow-up id} as in \cite{CheLin02sharp}. 
\begin{proof}
By \autoref{outside}, we have
\begin{equation}\label{differnce}
    \rho_\varepsilon=8\pi-\varepsilon+O\left(e^{-\lambda_\varepsilon}\right).
\end{equation}
This implies that we need to control $\rho_\varepsilon-8\pi$, which is equivalent to compute $-\int_{\partial B_{\delta_0}(p_\varepsilon)}\frac{\partial \eta_\varepsilon}{\partial\nu}d\sigma$ by \autoref{boundary term}. To do so, we set 
\begin{equation*}
    \psi=\dfrac{1-a|x-y_{\varepsilon}|^2}{1+a|x-y_{\varepsilon}|^2} \quad \text{for} \ x\in\mathbb{R}^2,
\end{equation*}
where $a=\frac{(8\pi-\varepsilon)h(p_\varepsilon)}{8}e^{\lambda_\varepsilon}$. Then $\psi$ satisfies
\begin{equation}\label{eq:psi}
\Delta_0\psi+(8\pi-\varepsilon)h(p_\varepsilon)e^{v_\varepsilon}\psi=0,
\end{equation}
where $\Delta_0$ is the standard Laplacian in $\mathbb{R}^2$.
On the other hand, by \eqref{differnce}, we have
\begin{equation}\label{eq:eta}
\begin{split}
    \Delta_0\eta_\varepsilon&=\Delta_0 u_\varepsilon-\Delta_0 v_\varepsilon-\Delta_0 G_{\varepsilon}^*\\
    &=-(8\pi-\varepsilon)h(p_\varepsilon)e^{v_\varepsilon(x)}H(x,\eta_\varepsilon)+O(e^{-\lambda_\varepsilon}),
    \end{split}
\end{equation}
where 
\begin{equation*}
    H(x,t)=\frac{h^*(x)}{h(p_\varepsilon)}e^{t+G_{\varepsilon}^*(x)-G_{\varepsilon}^*(0)}-1
\end{equation*}
and $h^*(x)=h(x)e^{2\phi(x)}$, $\phi(x)$ comes from the metric $ds^2=e^{2\phi(x)}dx^2$ with $\phi(0)=0$ and $\nabla\phi(0)=0$.
By using \eqref{eq:psi}, \eqref{eq:eta} and integration by parts, we get
\begin{equation*}
\begin{split}
    \int_{\partial B_{\delta_0}(p_\varepsilon)}\left(\psi\dfrac{\partial\eta_\varepsilon}{\partial\nu}-\eta_\varepsilon\dfrac{\partial\psi}{\partial\nu}\right)d\sigma&=\int_{B_{\delta_0}(p_\varepsilon)}(\psi\Delta_0\eta_\varepsilon-\eta_\varepsilon\Delta_0\psi)dx\\
    &=-\int_{B_{\delta_0}(p_\varepsilon)}\psi(x)(8\pi-\varepsilon)h(p_\varepsilon)e^{v_\varepsilon(x)}(H(x,\eta_\varepsilon)-\eta_\varepsilon(x))+O\left(e^{-\lambda_\varepsilon}\right).
    \end{split}
\end{equation*}
Since $\psi$ satisfies
\begin{equation*}
    \psi(x)=-1+\dfrac{2}{1+a|x-y_{\varepsilon}|^2}=-1+O\left(e^{-\lambda_\varepsilon}\right) \ \text{and} \ |\nabla\psi(x)|=O\left(e^{-\lambda_\varepsilon}\right)
\end{equation*}
 for $x\in\partial B_{\delta_0}(p_\varepsilon)$, we have
 \begin{equation*}
 \begin{split}
     -\int_{\partial B_{\delta_0}(p_\varepsilon)}\dfrac{\partial \eta_\varepsilon}{\partial\nu}d\sigma=-\int_{B_{\delta_0}(p_\varepsilon)}\psi(x)(8\pi-\varepsilon)h(p_\varepsilon)e^{v_\varepsilon(x)}(H(x,\eta_\varepsilon)-\eta_\varepsilon(x))+O\left(e^{-\lambda_\varepsilon}\right).
      \end{split}
 \end{equation*}

Recall
\begin{equation*}
\begin{split}
    H(x,\eta_\varepsilon)-\eta_\varepsilon(x)&=\dfrac{h^*(x)}{h(p_\varepsilon)}e^{\eta_\varepsilon+G_{\varepsilon}^*(x)-G_{\varepsilon}^*(0)}-1-\eta_\varepsilon(x)\\
    &=H(x,0)+H(x,0)\eta_\varepsilon+O(1)|\eta_\varepsilon|^2,
    \end{split}
\end{equation*}
where
\begin{equation*}
\begin{split}
    H(x,0)&=\frac{h^*(x)}{h(p_\varepsilon)}e^{G_{\varepsilon}^*(x)-G_{\varepsilon}^*(0)}-1\\
    &=\frac{1}{h(p_\varepsilon)}e^{2\phi(x)+\ln{h}(x)+G^*(x)-G^*(p_\varepsilon)}-1\\
    &=\langle b_\varepsilon,x\rangle+\langle B_\varepsilon x,x\rangle+O(1)|x|^{2+\beta},
    \end{split}
\end{equation*}
where $b_\varepsilon$ and $B_\varepsilon$ are the gradient and Hessian of $H(x,0)$ at $x=0$. By \autoref{cp position}, we have $|b_\varepsilon|=O\left(e^{-\lambda/2}\right)$.

Let $z$ and $z_\varepsilon$ satisfy
\begin{align*}
   \begin{cases}
   x=e^{-\frac{\lambda_\varepsilon}{2}}\left(\dfrac{h(p_\varepsilon)(8\pi-\varepsilon)}{8}\right)^{-\frac12}z,\\
   y_{\varepsilon}=e^{-\frac{\lambda_\varepsilon}{2}}\left(\dfrac{h(p_\varepsilon)(8\pi-\varepsilon)}{8}\right)^{-\frac12}z_\varepsilon.
    \end{cases}
\end{align*}
Then we get
\begin{align*}
    \left|\int_{ B_{\delta_0}(p_\varepsilon)}e^{v_\varepsilon}\langle b_\varepsilon,x\rangle dx\right|\leq& Ce^{-\lambda_\varepsilon}\int_{|z|\leq R_0}\left(1+|z-z_\varepsilon|^2\right)^{-2}|z|dz=O\left(e^{-\lambda_\varepsilon}\right),\\
    \int_{ B_{\delta_0}(p_\varepsilon)}e^{v_\varepsilon}|x|^{2+\beta} dx\leq& Ce^{-\frac{2+\beta}{2}\lambda_\varepsilon}\int_{|z|\leq R_0}\left(1+|z|^2\right)^{-2}|z|^{2+\beta}dz=O\left(e^{-\lambda_\varepsilon}\right)
\end{align*}
and
\begin{equation*}
\begin{split}
    \int_{ B_{\delta_0}(p_\varepsilon)}e^{v_\varepsilon}(x_\alpha-p_{\varepsilon,\alpha})(x_\beta-p_{\varepsilon,\beta}) dx=&\left((8\pi-\varepsilon)\frac{h(p_\varepsilon)}{8}\right)^{-2}e^{-\lambda_\varepsilon}\int_{|z|\leq R_0}\left(1+|z-z_\varepsilon|^2\right)^{-2}z_\alpha z_\beta dz\\
    &=\left((8\pi-\varepsilon)\frac{h(p_\varepsilon)}{8}\right)^{-2}e^{-\lambda_\varepsilon}\pi\left[\delta_{\alpha\beta}\ln{R_{\varepsilon}}+O\left(e^{-\frac{\lambda_\varepsilon}{2}}\right)\right],
    \end{split}
\end{equation*}
where $x_\alpha$ stands for the $\alpha$-th coordinate of $x$ and $1\leq\alpha,\beta\leq2$.
Putting those estimates above together, we have
\begin{equation*}
\begin{split}
    \int_{ B_{\delta_0}(p_\varepsilon)}(8\pi-\varepsilon)h(p_\varepsilon)e^{v_\varepsilon}H(x,0)dx&=\frac{32\pi}{(8\pi-\varepsilon)h(p_\varepsilon)}\left(B_\varepsilon^{11}+B_\varepsilon^{22}\right)e^{-\lambda_\varepsilon}\lambda_\varepsilon+O(1)e^{-\lambda_\varepsilon}.
    \end{split}
\end{equation*}
Note that $\Delta_0 G_{\varepsilon}^*(0)=\rho_\varepsilon=(8\pi-\varepsilon)+O\left(e^{-\lambda_\varepsilon}\right)$ and $-\Delta_0\phi(0)=K(p_\varepsilon)$. By \autoref{cp position}, we know
\begin{equation*}
\begin{split}
    B_\varepsilon^{11}+B_\varepsilon^{22}&=\dfrac12\Delta_0 H(0,0)\\
    &=\frac12(\Delta\ln{h}(p_\varepsilon)+8\pi-\varepsilon-2K(p_\varepsilon))+O\left(e^{-\lambda_\varepsilon}\right).
    \end{split}
\end{equation*}
For the remainder terms, we use \autoref{remainder} to get
\begin{equation*}
\begin{split}
    &\int_{B_{\delta_0}(p_{\varepsilon})} e^{v_\varepsilon}H(x,0)\eta_\varepsilon(x)dx=O\left(e^{-\lambda_\varepsilon}\right)\\
    &\int_{B_{\delta_0})(p_{\varepsilon})} e^{v_\varepsilon}\eta_\varepsilon^2(x)dx=O\left(e^{-\lambda_\varepsilon}+e^{-\tau\lambda_\varepsilon}|8\pi-\rho_\varepsilon|\right).
    \end{split}
\end{equation*}
    Therefore, 
    \begin{equation*}
        \rho_\varepsilon-8\pi=\dfrac{16\pi}{(8\pi-\varepsilon)h(p_\varepsilon)}\left[\Delta \ln h(p_\varepsilon)+8\pi-2K(p_\varepsilon)\right]\lambda_{\varepsilon}e^{-\lambda_{\varepsilon}}+O\left(e^{-\lambda_{\varepsilon}}\right)
    \end{equation*}
    and this completes the proof.
\end{proof}

\section{Proof of Theorem \ref{functional value}}

\begin{proof}On one hand, checking the proof in  \cite[Theorem 1.2]{SunZhu20global} step by step, we have
\begin{align}\label{eq:lower-functional}
\begin{split}
    \inf_{u\in H^1\left(\Sigma\right)}J_{8\pi}(u)=\lim_{\varepsilon\to 0}J_{8\pi}\left(u_{\varepsilon}\right)&\geq-1-\ln\pi-\left(\ln h(p_0)+\dfrac12 A(p_0)\right)\\
    &\geq-1-\ln\pi-\max_{p\in\Sigma}\left(\ln h^+(p)+\dfrac12 A(p)\right).
    \end{split}
\end{align}
We sketch the proof here. Without loss of generality, up to a conformal change of the metric, we may assume that the metric is the Euclidean metric around $p_0$ and we also assume $p_0$ is the origin $o\in\mathbb{B}\subset\Sigma$.  Choose $p_\varepsilon\to p_0$ such that
\begin{align*}
    \lambda_{\varepsilon}=u_{\varepsilon}\left(p_\varepsilon\right)=\max_{\Sigma}u_{\varepsilon}\to+\infty.
\end{align*}
Set $r_{\varepsilon}=e^{-\lambda_{\varepsilon}/2}$ and 
\begin{align*}
    \tilde u_{\varepsilon}=u_{\varepsilon}\left(p_\varepsilon+r_{\varepsilon}x\right)+2\ln r_{\varepsilon},\quad\abs{x}< r_{\varepsilon}^{-1}\left(1-\abs{p_\varepsilon}\right).
\end{align*}
Then $\tilde u_{\varepsilon}$ converges to $w$ in $C^{\infty}_{loc}\left(\mathbb{R}^2\right)$ where 
\begin{align*}
    w(x)=-2\ln\left(1+\pi h(p_0)\abs{x}^2\right).
\end{align*}
We denote by $o_{\varepsilon}(1)$ (resp. $o_{R}(1), o_{\delta}(1)$) the terms which tents to zero as $\varepsilon\to0$ (resp. $R\to\infty, \delta\to0$). Moreover, $o_{\varepsilon}(1)$ may depend on $R,\delta$, while $o_{R}(1)$ may depend on $\delta$.  We have
\begin{align*}
    \dfrac{1}{16\pi}\int_{\mathbf{B}_{r_{\varepsilon}R}(p_\varepsilon)}\abs{\nabla u_{\varepsilon}}^2=\dfrac{1}{16\pi}\int_{\mathbf{B}_{R}}\abs{\nabla\tilde{u}_{\varepsilon}}^2=\ln\left(\pi h(p_0)R^2\right)-1+o_{\varepsilon}(1)+o_{R}(1).
\end{align*}
According to \autoref{convergence}, a direct calculation yields
\begin{align*}
    \dfrac{1}{16\pi}\int_{\Sigma\setminus\mathrm{B}_{\delta}(p_{\varepsilon})}\abs{\nabla u_{\varepsilon}}^2=-2\ln\delta+\dfrac12 A(p_0)+o_{\varepsilon}(1)+o_{\delta}(1).
\end{align*}
Under polar coordinates $(r,\theta)$, set
\begin{align*}
    u^*_{\varepsilon}(r)=\dfrac{1}{2\pi}\int_{0}^{2\pi} u_{\varepsilon}\left(p_{\varepsilon}+re^{\sqrt{-1}\theta}\right)\dif\theta.
\end{align*}
Then
\begin{align*}
    u^*_{\varepsilon}(\delta)=&\int_{\Sigma}u_{\varepsilon}-4\ln\delta+ A(p_0)+o_{\varepsilon}(1)+o_{\delta}(1),\\
    u^*_{\varepsilon}\left(r_{\varepsilon}R\right)=&-2\ln r_{\varepsilon}-2\ln\left(\pi h(p_0)R^2\right)+o_{\varepsilon}(1)+o_{R}(1).
\end{align*}
Solve
\begin{align*}
\begin{cases}
 -\Delta \xi_{\varepsilon}=0,&\text{in}\ \mathbf{B}_{\delta}\left(p_{\varepsilon}\right)\setminus\mathbf{B}_{r_{\varepsilon}R}\left(p_{\varepsilon}\right),\\
 \xi_{\varepsilon}=u_{\varepsilon}^*,&\text{on}\ \partial\left(\mathbf{B}_{\delta}\left(p_{\varepsilon}\right)\setminus\mathbf{B}_{r_{\varepsilon}R}\left(p_{\varepsilon}\right)\right).
\end{cases}
\end{align*}
We have
\begin{align*}
\begin{split}
    \dfrac{1}{16\pi}\int_{\mathbf{B}_{\delta}\left(p_{\varepsilon}\right)\setminus\mathbf{B}_{r_{\varepsilon}R}\left(p_{\varepsilon}\right)}\abs{\nabla u_{\varepsilon}}^2&\geq\dfrac{1}{16\pi}\int_{\mathbf{B}_{\delta}\left(p_{\varepsilon}\right)\setminus\mathbf{B}_{r_{\varepsilon}R}\left(p_{\varepsilon}\right)}\abs{\nabla u^*_{\varepsilon}}^2\\
    &\geq\dfrac{1}{16\pi} \int_{\mathbf{B}_{\delta}\left(p_{\varepsilon}\right)\setminus\mathbf{B}_{r_{\varepsilon}R}\left(p_{\varepsilon}\right)}\abs{\nabla\xi_{\varepsilon}}^2=\dfrac{\left(u_{\varepsilon}^*(\delta)-u_{\varepsilon}^*(r_{\varepsilon}R)\right)^2}{8\left(\ln\delta-\ln\left(r_{\varepsilon}R\right)\right)}.
    \end{split}
\end{align*}
Thus
\begin{align*}
\begin{split}
    \dfrac{1}{16\pi}\int_{\mathbf{B}_{\delta}\left(p_{\varepsilon}\right)\setminus\mathbf{B}_{r_{\varepsilon}R}\left(p_{\varepsilon}\right)}\abs{\nabla u_{\varepsilon}}^2\geq&\dfrac{\left(u_{\varepsilon}^*(\delta)-u_{\varepsilon}^*(r_{\varepsilon}R)\right)^2}{-8\ln r_{\varepsilon}}\left(1+\dfrac{\ln\left(R/\delta\right)}{-\ln r_{\varepsilon}}\right)\\
    =&\dfrac{\left(\tau_{\varepsilon}+\int_{\Sigma}u_{\varepsilon}-2\ln r_{\varepsilon}\right)^2}{-8\ln r_{\varepsilon}}+\dfrac{1}{8}\left(2+\dfrac{\tau_{\varepsilon}}{\ln r_{\varepsilon}}+\dfrac{\int_{\Sigma}u_{\varepsilon}}{\ln r_{\varepsilon}}\right)^2\ln(R/\delta)\\
    &-\int_{\Sigma}u_{\varepsilon}-4\ln\left(R/\delta\right)-A(p_0)-2\ln(\pi h(p_0))\\
    &+o_R(1)+o_\delta(1),
    \end{split}
\end{align*}
where  
\begin{align*}
\begin{split}
    \tau_{\varepsilon}&=u^*_{\varepsilon}(\delta)-u^*_{\varepsilon}\left(r_{\varepsilon}R\right)-\int_{\Sigma}u_{\varepsilon}+2\ln r_{\varepsilon}\\
    &=4\ln\left(R/\delta\right)+ A(p_0)+2\ln\left(\pi h(p_0)\right)+o_{\varepsilon}(1)+o_{\delta}(1)+o_{R}(1).
    \end{split}
\end{align*}
Hence, we get
\begin{align*}
\begin{split}
    C\geq J_{8\pi}\left(u_{\varepsilon}\right)\geq&-1-\ln\pi-\ln h(p_0)-\dfrac12A(p_0)\\
    &+\dfrac{\left(\tau_{\varepsilon}+\int_{\Sigma}u_{\varepsilon}-2\ln r_{\varepsilon}\right)^2}{-8\ln r_{\varepsilon}}+\dfrac{1}{8}\left(\left(2+\dfrac{\tau_{\varepsilon}}{\ln r_{\varepsilon}}+\dfrac{\int_{\Sigma}u_{\varepsilon}}{\ln r_{\varepsilon}}\right)^2-16\right)\ln(R/\delta)\\
    &+o_{\varepsilon}(1)+o_{R}(1)+o_{\delta}(1)
    \end{split}
\end{align*}
which implies
\begin{align*}
    \int_{\Sigma}u_{\varepsilon}=-\lambda_{\varepsilon}+O\left(\sqrt{\lambda_{\varepsilon}}\right)
\end{align*}
and we obtain \eqref{eq:lower-functional}.

On the other hand, checking the proof in \cite[Theorem 1.2]{DinJosLiWan97differential}  step by step, for each $p$ with $h(p)>0$, there exists a sequence $\phi_{\varepsilon}\in H^1\left(\Sigma\right)$ such that
\begin{align*}
\begin{split}
    J_{8\pi}\left(\phi_{\varepsilon}\right)=&-1-\ln\pi-\left(\ln h(p)+\dfrac12A(p)\right)\\
    &-\dfrac{1}{4}\left(\Delta\ln h(p)+8\pi-2K(p)+\abs{\nabla\left(\ln h+\dfrac12A\right)(p)}^2\right)\varepsilon\ln\varepsilon^{-1}\\
    &+o\left(\varepsilon\ln\varepsilon^{-1}\right).
    \end{split}
\end{align*}
Here we used the fact that the Green function $G$ is symmetric. 
These test functions $\phi_{\varepsilon}$ can be constructed as following: without loss of generality, assume $p=0$ and 
\begin{align*}
    8\pi G(x,0)=&-2\ln\abs{x}+A(p)+b_1x_1+b_2x_2+\beta(x),
\end{align*}
and take
\begin{align*}
    \phi_{\varepsilon}(x)=\begin{cases}
    -2\ln\left(\abs{x}^2+\varepsilon\right)+b_1x_1+b_2x_2+\ln\varepsilon,&\abs{x}<\alpha_{\varepsilon}\sqrt{\varepsilon},\\
    8\pi G(x,0)-\eta\left(\alpha_{\varepsilon}\sqrt{\varepsilon}\abs{x}\right)\beta(x)+C_{\varepsilon}+\ln\varepsilon,&\alpha_{\varepsilon}\sqrt{\varepsilon}\leq\abs{x}<2\alpha_{\varepsilon}\sqrt{\varepsilon},\\
    8\pi G(x,0)+C_{\varepsilon}+\ln\varepsilon,&\abs{x}\geq 2\alpha_{\varepsilon}\sqrt{\varepsilon},
    \end{cases}
\end{align*}
where $\eta$ is a cutoff function supported in $[0,2]$ and $\eta=1$ on $[0,1]$ and the positive constants $\alpha_{\varepsilon}$ and $C_{\varepsilon}$ are chosen carefully. The assumption $h$ is positive in \cite{DinJosLiWan97differential} is used only to ensure that
\begin{align*}
    \lim_{\varepsilon\searrow0}\int_{\Sigma}he^{\phi_{\varepsilon}}>0.
\end{align*}
If $p$ is a critical point (e.g., a maximum point) of the function $2\ln h^++A$, then
\begin{align*}
    J_{8\pi}\left(\phi_{\varepsilon}\right)=&-1-\ln\pi-\left(\ln h(p)+\dfrac12A(p)\right)-\dfrac{1}{4}\left(\Delta\ln h(p)+8\pi-2K(p)\right)\varepsilon\ln\varepsilon^{-1}+o\left(\varepsilon\ln\varepsilon^{-1}\right).
\end{align*}
  This gives
\begin{align*}
    \inf_{u\in H^1\left(\Sigma\right)}J_{8\pi}(u)=-1-\ln\pi-\max_{p\in\Sigma}\left(\ln h^+(p)+\dfrac12A(p)\right)=-1-\ln\pi-\left(\ln h(p_0)+\dfrac12 A(p_0)\right).
\end{align*}
In particular, the blow-up point $p_0$ must be a maximum point of the function $\ln h^++A$. 

\end{proof}

\begin{rem}
One can write down the $o_\varepsilon(1)$ as follows. By  \autoref{estimate B} and \eqref{expansion G}, direct computations give us
\begin{align*}
\begin{split}
    \dfrac{1}{16\pi}\int_{\Sigma\setminus{B}_{\delta}(p_{\varepsilon})}\abs{\nabla u_{\varepsilon}}^2=&\left(1-\dfrac{\varepsilon}{4\pi}+\frac{\varepsilon^2}{64\pi^2}+O\left(e^{-\lambda_\varepsilon}\right)\right)\left(-2\ln\delta+\dfrac12 A(p_\varepsilon)+O\left(e^{-\lambda_\varepsilon}\right)+o_{\delta}(1)\right)+O\left(e^{-\lambda_\varepsilon}\right)\\
    =&-2\ln\delta+\dfrac12 A(p_\varepsilon)
    -\dfrac{\varepsilon}{4\pi}\left(-2\ln\delta+\dfrac12 A(p_\varepsilon)+O\left(e^{-\lambda_\varepsilon}\right)+o_{\delta}(1)\right)\\
     &\quad+O\left(\varepsilon^2\right)+O\left(e^{-\lambda_\varepsilon}\right)+o_{\delta}(1).
    \end{split}
\end{align*}
From the proof of \autoref{blow-up id}, we also get the following

\begin{align*}
    \int_{\mathrm{B}_{\delta}(p_{\varepsilon})}\abs{\nabla \eta_{\varepsilon}}^2=O\left(\varepsilon^2\delta\right)+O\left(e^{-\lambda_\varepsilon}\right),
\end{align*}
\begin{align*}
    \dfrac{1}{16\pi}\int_{{B}_{r_{\varepsilon}R}(p_\varepsilon)}\abs{\nabla v_{\varepsilon}}^2=\ln\left(\pi h(p_0)R^2\right)-1+o_{R}(1),
\end{align*}
\begin{align*}
 \int_{B_{\delta}(p_{\varepsilon})}\abs{\nabla  G^*}^2=O\left(\delta^2\right)
\end{align*}
and
\begin{align*}
\dfrac{1}{16\pi}\int_{{B}_{r_{\varepsilon}R}(p_\varepsilon)}\abs{\nabla  G^*}^2=O\left(r_{\varepsilon}^2\right)=O(e^{-\lambda_\varepsilon}).
\end{align*}
These imply that
\begin{align*}
    \dfrac{1}{16\pi}\int_{B_{r_{\varepsilon}R}(p_{\varepsilon})}\abs{\nabla u_{\varepsilon}}^2=\ln\left(\pi h(p_0)R^2\right)-1+O\left(\varepsilon^2 e^{-\frac{\lambda_\varepsilon
}{2}}\right)+O\left(e^{-\lambda_\varepsilon}\right)+o_{R}(1).
\end{align*}
On the neck, $o_\varepsilon(1)$ are the convergent rates in \autoref{convergence} and $\tilde u_{\varepsilon}\to w$.
\end{rem}

\biboptions{longnamesfirst,sort&compress}
\bibliographystyle{elsarticle-harv}

\bibliography{kw}

\end{document}